\documentclass[12pt]{amsart}
\usepackage{amssymb}
\usepackage{amsthm}
\usepackage{amsmath}
\usepackage{bm}
\usepackage{stackrel}
\usepackage{enumerate}

\theoremstyle{plain}
\newtheorem{Thm}{Theorem}[section]
\newtheorem{Lem}[Thm]{Lemma}
\newtheorem{Prop}[Thm]{Proposition}
\newtheorem{Cor}[Thm]{Corollary}

\theoremstyle{definition}
\newtheorem{remark}[Thm]{Remark}

\numberwithin{equation}{section}

\newcommand{\suppress}[1]{}

\newcommand{\CC}{{\mathbb C}}

\newcommand{\RR}{{\mathbb R}}

\DeclareMathOperator{\res}{res}
\DeclareMathOperator{\Span}{span}

\newcommand\re{\Re}

\newcommand\dist{{\mathop{\mbox{\rm dist}}}}

\def\<{\langle}
\def\>{\rangle}

\begin{document}

\title{Zero-free regions for Dirichlet series (II)}

\author{C. Delaunay}
\author{E. Fricain}
\author{E. Mosaki}
\author{O. Robert}
\address{Christophe Delaunay, Universit\'e de Franche-Comt\'e; Laboratoire de Math\'ematiques de Besan\c con; CNRS UMR 6623; 16, route de Gray, F-25030 Besanon, France}
\email{delaunay@univ-fcomte.fr}
\address{Emmanuel Fricain, Universit\'e de Lyon; Universit\'e Lyon 1; Institut Camille Jordan CNRS UMR 5208; 43, boulevard du 11 Novembre 1918, F-69622 Villeurbanne}
\email{fricain@math.univ-lyon1.fr}
\address{Elie Mosaki, Universit\'e de Lyon; Universit\'e Lyon 1; Institut Camille Jordan CNRS UMR 5208; 43, boulevard du 11 Novembre 1918, F-69622 Villeurbanne}
\email{mosaki@math.univ-lyon1.fr}
\address{Olivier Robert, Universit\'e de Lyon; F-42023, Saint-Etienne, France; Universit\'e de Saint-Etienne; F-42000, Saint-Etienne, France; Laboratoire de math\'ematiques (LAMUSE, EA 3989); 23, rue du Dr P. Michelon, F-42000, Saint-Etienne France;}
\email{olivier.robert@univ-st-etienne.fr}
\date{}
\thanks{This work was supported by the ANR project no. 07-BLAN-0248 "ALGOL", the ANR project no. 09-BLAN-005801 "FRAB" and the ANR project no. 08-BLAN-0257 "PEPR".}



\begin{abstract} 
In this paper, we continue some work devoted to explicit zero-free discs for a large class of  Dirichlet series. In a previous article, such zero-free regions were described using some spaces of functions which were defined with some technical conditions. Here we give two different natural ways in order to remove those technical conditions. In particular this allows to right down explicit zero-free regions differently and to obtain for them an easier description useful for direct applications.
\end{abstract}

\maketitle

\section{Introduction}
As usual, we denote by $\{t\}$ the fractional part of the real number $t$. We let $\mathcal B^\sharp$ be the closed subspace of $L^2(0,+\infty)$ spanned by functions of the form 
\begin{equation}\label{eq:function-beurling}
f \colon t \longmapsto \sum_{k=1}^{n}c_k \left\{\frac{\alpha_k}{t}\right\},
\end{equation}
where $c_k\in\CC$ and $0<\alpha_k\leq 1$ are restricted to the condition
\begin{equation}\label{eq:admissibilite-beurling}
\sum_{k=1}^n c_k\alpha_k=0.
\end{equation}
A. Beurling and B. Nyman (see \cite{Beurling}, \cite{Nyman}) proved that the Riemann zeta function does not vanish on the half-plane $\Re(s)>1/2$ if and only if 
$\chi_{(0,1)} \in \mathcal B^\sharp$, where $\chi_{(0,1)}$ is the characteristic function of the interval~$(0,1)$. Their result is known as the Beurling-Nyman criterion for the Riemann hypothesis. This theorem was extended by A. de Roton in \cite{anne-TAMS}  in the case of  $L$-functions in the Selberg class. In \cite{Nikolski-AIF}, N. Nikolski obtained an explicit version for the Beurling-Nyman's criterion in the case of the Riemann zeta function. Similarly,  in 
\cite{dfmr} an extended explicit version had been given for a large class of Dirichlet series (which include largely the Selberg class).
In all these previous works, some spaces of functions, generalizing ${\mathcal B}^\sharp$, have to be considered and their definitions involve several technical
conditions of the same type as \eqref{eq:admissibilite-beurling}. These conditions appear naturally in order to control the pole, coming from $L(s)$ at $s=1$, of some auxiliary functions. 
~\\
The fact is that these conditions are useless if we are  interested in an equivalent criterion for the (generalized) Riemann hypothesis.
Indeed, for the Riemann zeta function, it is proved  in \cite{Duarte3} that we can omit the condition \eqref{eq:admissibilite-beurling}: let ${\mathcal B}$ be the closed subspace of $L^2(0,+\infty)$ spanned by functions of the form
\begin{equation*}
f(t)=\sum_{k=1}^{n}c_k \left\{\frac{\alpha_k}{t}\right\},\qquad (t>0).
\end{equation*} 
Then the zeta function does not vanish on the half-plane $\Re(s)>1/2$ if and only if $\chi_{(0,1)} \in {\mathcal B}$. Furthermore, for $0<\lambda\leq 1$, if ${\mathcal B_\lambda}$ denotes the subspace of ${\mathcal B}$ formed by functions $f$ such that $\min_{1\leq k\leq n}\alpha_k\geq\lambda$, then the authors in \cite{Duarte3} also proved that there exists a constant $C>0$ such that 
\begin{equation}\label{liminf}
\liminf_{\lambda\to 0}d(\lambda)\sqrt{\log(1/\lambda)}\geq C,
\end{equation}
where $d(\lambda)$ denotes the distance between $\chi$ and ${\mathcal B_\lambda}$. Such results were also generalized in \cite{Anne-BSMF} and \cite{Anne-JNT} for the Selberg class.\bigskip
~\\
In this article, we explain how to  drop off the conditions of type~\eqref{eq:admissibilite-beurling}  used in \cite{dfmr}.  On the one hand, we give a Beurling-Nyman criterion of the same type of \cite{Duarte3} and \cite{Anne-BSMF} but for a wide class of Dirichlet series (we do not need any Euler product nor functional equation). And on the other hand, we also obtain explicit zero free regions of the same shape of \cite{dfmr} without the technical conditions. In particular, these give new explicit zero free regions that are easier to deal with. For these purposes we will give two different and independent (but complementary) methods.

\section{Notation} \label{section_2}


In this section, we will give some notation and recall some results that were obtained in \cite{dfmr} (we will refer to this article several times). 
For $r \in \RR$, we denote by $\Pi_r$ the half-plane 
$$
\Pi_r = \{ s \in \CC \; : \; \re(s)>r\}.
$$
We fix a Dirichlet series $L(s)=\sum_{n\geq 1}\frac{a_n}{n^s}$ satisfying the following conditions:
\begin{itemize}
\item For every $\varepsilon >0$, we have $a_n = O_\varepsilon(n^\varepsilon)$. 
\item There exists $\sigma_0<1$ such that the function $s \mapsto L(s)$ admits a meromorphic continuation to $\re (s)>\sigma_0$ with a unique pole of order $m_L$ at $s=1$.
\item The function $s \mapsto (s-1)^{m_L} L(s)$ is analytic with finite order in~$\Pi_{\sigma_0}$.
\end{itemize}
The growth condition on the coefficients $(a_n)_n$ implies that $L(s)$ is an absolutely convergent Dirichlet series for $\re(s)>1$. We also consider a function 
$\varphi:[0,+\infty[ \longrightarrow\CC$ such that 
\begin{itemize}
\item $\varphi$ is supported on $[0,1]$ and is locally bounded on $(0,1)$. 
 \item $\varphi(x)=O(x^{-\sigma_0})$ when $x\to 0$.
\item $\varphi(x)=O((1-x)^{-\sigma_1})$ when $x\to 1^-$, for some $\sigma_1<1/2$.
\end{itemize}
~\\
We recall that the (unnormalized) Mellin transform of a  Lebesgue-measurable function $\varphi:[0,+\infty[\to\CC$  is the function $\widehat\varphi$ defined by
$$
\widehat{\varphi} (s) = \int_0^{+\infty} \varphi(t) t^s \frac{dt}{t}  \qquad (s \in \CC),
$$ 
whenever the integral is absolutely convergent. If $\varphi$ satisfies the conditions above, we easily see that $s\longmapsto \hat\varphi(s)$ is analytic on $\Pi_{\sigma_0}$. The normalized Mellin transform  $\mathcal{M}:\varphi\mapsto \frac{1}{\sqrt{2\pi}} \widehat{\varphi}$ is a unitary operator that maps the space $ L^2_*\left((0,1), \frac{dt}{t^{1-2\sigma}}\right)$ onto    $H^2(\Pi_\sigma)$, where $ L^2_*\left((0,1), \frac{dt}{t^{1-2\sigma}}\right)$ is  the subspace of functions in $L^2\left((0,+\infty), \frac{dt}{t^{1-2\sigma}}\right)$  that vanish almost everywhere on $(1,+\infty)$, and     $H^2(\Pi_\sigma)$ is the Hardy space   of analytic functions $f$ : $\Pi_\sigma \rightarrow \CC$
such that $\Vert f \Vert_2 < \infty$ with
\begin{equation}\label{norm_H2}
\Vert f \Vert_2 = \sup_{x > \sigma} \left( \int_{-\infty}^{+\infty} \vert f(x +it)\vert^2 dt \right)^{\frac 12}.
\end{equation}
 We also recall that  $\mathcal M$  extends to a unitary operator from the space $L^2((0,+\infty),\frac{du}{u^{1-2\sigma}})$ onto $L^2(\sigma+i\RR)$ (use the Fourier--Plancherel's theorem and the change of variable going from the Fourier transform to the Mellin transform). 
With our choices of $L$ and $\varphi$ we define
\begin{equation}\label{psi}
\psi(u)=\res\left(L(s)\hat\varphi(s)u^s,s=1\right)-\sum_{n <u }a_n\varphi\left(\frac nu\right) \qquad (u\in\RR_+^*),
\end{equation}
where $\res(F(s),s=1)$ denotes the residue of the meromorphic function $F$ at $s=1$. We recall that by definition of $\psi$, there exists  $(p_0,p_1,\dots,p_{m_L-1})\in \CC^{m_L}$ with  $p_{m_L-1}\neq 0$ such that for $0<u<1$
\begin{equation}\label{psi-u-petit}
\psi(u)=u\sum_{k=0}^{m_L-1}p_k(\log u)^k\qquad (0 < u<1).
\end{equation}
Indeed, since the function $s\mapsto L(s)\hat\varphi(s)$ has a pole of order $m_L$ at $s=1$, we can write
\begin{equation}\label{eq:defnH}
L(s)\hat\varphi(s)=\sum_{k=0}^{m_L-1}\frac{k!p_k}{(s-1)^{k+1}}-H(s)\qquad (s\in\Pi_{\sigma_0}),
\end{equation}
with $p_{m_L-1}\neq 0$ and where $H$ is some analytic function in $\Pi_{\sigma_0}$. For each $0\leq k\leq m_L-1$, we have
\[
\hbox{res}\left(\frac{u^s}{(s-1)^{k+1}},s=1\right)=\frac{u(\log u)^{k}}{k!},
\]
which gives 
\begin{equation}\label{eq:res}
\hbox{res}\left(L(s)\hat\varphi(s)u^s,s=1\right)=u \sum_{k=0}^{m_L-1}p_k(\log u)^k.
\end{equation}
Now \eqref{psi-u-petit} follows from $\psi(u)=\hbox{res}\left(L(s)\hat\varphi(s)u^s,s=1\right)$ if $0<u<1$. 
~\\
Hence, it is clear that for $r >\sigma_0$, the function $t\mapsto t^{r-\sigma_0}\psi(\tfrac{1}{t})$ belongs to $ L^{2}\big((0,+\infty), \tfrac{dt}{t^{1-2\sigma_0}} \big)$ if and only if
\begin{equation}\label{psi-in-L2}
r<1\mbox{ and }\int_{1}^{+\infty}|\psi(t)|^2\frac{dt}{t^{1+2r}}<+\infty.
\end{equation}
By \cite[Theorem 2.1]{dfmr} this is equivalent to $r<1$ and the fact that the function $t\longmapsto L(r+it)\hat\varphi(r+it)$ belongs to $L^2(\RR)$.
In the classical examples such as the Selberg class, such a real number $r$ exists, and moreover, each $r'\in [r,1)$ also satisfies \eqref{psi-in-L2}.\bigskip
~\\
{\bf In the sequel, we assume that there exists  $r_0>\sigma_0$ satisfying~\eqref{psi-in-L2} and we fix $r_0$ once and for all.}\bigskip 
~\\
We set 
$$
\mathcal{S}:=\bigcup_{\ell\ge 1}(0,1]^{\ell}\times \CC^{\ell}.
$$
Each  $A\in \mathcal{S}$ is a couple $(\alpha,c)$ where $\alpha=(\alpha_1,\dots,\alpha_{\ell})\in (0,1]^{\ell} $ and $c=(c_1,\dots,c_{\ell})\in  \CC^{\ell}$ for some $\ell\ge 1$. That $\ell$ is called the \textsl{length of} $A$ and is noted $\ell(A)$.
~\\
A sequence $A=(\alpha,c)\in \mathcal{S}$ is called $m$-admissible\footnote{These are exactly the conditions we mentioned in the introduction.}  if
$$
\sum_{j=1}^{\ell(A)}c_j\alpha_j(\log \alpha_j)^{k} = 0 \mbox{ for any }0\le k\le m-1.
$$
We denote by $\mathcal{S}^{\sharp}$ the subset of the sequences $A\in \mathcal{S}$ that are $m_L$-admissible. To each  $A=(\alpha,c)\in \mathcal{S}$, we associate the function $f_{A,r}$ defined by
\begin{equation}\label{definition_fA}
f_{A,r}(t):=t^{r-\sigma_0}\sum_{j=1}^{\ell(A)}c_j\psi\left(\frac{\alpha_j}{t}\right)\qquad (t>0).
\end{equation}
Then for $r_0\leq r<1$ and for $A\in \mathcal{S}$, we have $f_{A,r}\in  L^{2}\big((0,+\infty), \tfrac{dt}{t^{1-2\sigma_0}} \big)$. Futhermore, if $A\in \mathcal{S}^\sharp$, then the function $f_{A,r}$ is identically zero on $(1,+\infty$) (see \cite[Theorem 4.3]{dfmr}).
\medskip
~\\
We set
\begin{equation}\label{Kr}
K_r:=\mathrm{Span}\{f_{A,r}\colon A \in \mathcal{S}\}\qquad (r_0\leq r<1)
\end{equation}
and
\begin{equation}\label{Kr-sharp}
K_r^{\sharp}:=\mathrm{Span}\{f_{A,r}\colon  A \in \mathcal{S}^{\sharp}    \}\qquad (r_0\leq r<1).
\end{equation}
Here the (closed) span are taken with respect to $L^{2}\big((0,+\infty), \tfrac{dt}{t^{1-2\sigma_0}} \big)$.
For $\lambda\in \Pi_{\sigma_0}$, we set
$$
w_{\lambda}(t):=t^{\overline{\lambda}-2\sigma_0}\chi_{(0,1)}(t)\qquad (t>0)
$$
and for $r_0\leq r < 1$  we let
\begin{equation}\label{definition_dr}
d_{r}(\lambda):=\mathrm{dist}(w_{\lambda}, K_{r}) \qquad \mbox{ and } \qquad d_{r}^{\sharp}(\lambda):=\mathrm{dist}(w_{\lambda}, K_{r}^{\sharp}).
\end{equation}
Since $K_r^{\sharp}\subset K_r$, it is immediate that 
\begin{equation}\label{eq1:dr-drsharp}
d_{r}(\lambda)\le d_{r}^{\sharp}(\lambda)\quad \big(r_0 \leq r <1,\thinspace \lambda\in \Pi_{\sigma_0}\big).
\end{equation}
We can now state Theorem 2.2 of \cite{dfmr}\footnote{The reader may be careful that in \cite{dfmr} the subspace $K_r^\sharp$ and the distance  $d_r^\sharp(\lambda)$ were denoted by $K_r$ and $d_r(\lambda)$.} :
\begin{Thm}\label{sans_zeros_avec_conditions}
Let $\lambda \in \Pi_{\sigma_0}$. Then the function $L$ does not vanish on $r-\sigma_0 + D^\sharp_r(\lambda)$ where
$$
D^\sharp_r(\lambda) := \left\{ \mu \in \CC \colon \left|\frac{\mu - \lambda}{\mu + \overline{\lambda}-2\sigma_0}\right| < \sqrt{1-2(\Re(\lambda) -\sigma_0){d_r^\sharp}^2(\lambda)} \right\} .
$$
\end{Thm}
\bigskip
~\\
It is also obtained the following result (see \cite[Theorem 2.4]{dfmr}):
\begin{Thm}\label{avec_conditions}
Suppose that the function $\hat\varphi$ does not vanish on the half-plane $\Pi_{r}$, that 
$\limsup_{x\to+\infty}\frac{\log|\hat\varphi(x+r-\sigma_0)|}{x} = 0$ and that $a_1 \neq 0$. Then the following assertions are equivalent:
\begin{enumerate}
 \item The function $L$  does not vanish on the half-plane $\Pi_{r}$.
\item There exists $\lambda\in\Pi_{\sigma_0}$ such that  $d_r^\sharp(\lambda)=0$.
\item For all $\lambda\in\Pi_{\sigma_0}$, we have $d_r^\sharp(\lambda)=0$.
\item We have $K^\sharp_r = L^2_*((0,1),dt/t^{1-2\sigma_0})$.
\end{enumerate}
\end{Thm}
~\\ 
This last theorem is exactly a Beurling-Nyman's criterion for $L$. The key point in the proof of these two results is the fact that    the Mellin transform of each $f_{A,r}\in K_r^{\sharp}$ is the product of $L(s)\hat\varphi(s)$ with a suitable function $g_A(s)$ that kills the pole at $s=1$. In that case, the function $L(s)\hat\varphi(s)g_A(s)$   belongs to the Hardy space $H^2(\Pi_r)$, and we may use the theory of analytic reproducing kernel Hilbert spaces.

\medskip
In this paper, we are interested with the following question: is it possible to replace the distance $d_r^\sharp(\lambda)$ by $d_r(\lambda)$ and the space $K_r^\sharp $ by $K_r$ in both previous results? Of course  if $m_L=0$, then $\mathcal S=\mathcal S^\sharp$ and $K_r=K_r^\sharp$ and there is nothing to do! So we assume in the following that $m_L\geq 1$. When we replace $K_r^\sharp$ by $K_r$, the pole at $s=1$ coming from the Dirichlet series is no longer compensated. In particular, for $A\in\mathcal S\setminus\mathcal S^\sharp$, the function $L(s)\hat \varphi(s)g_A(s)$ does not belong to the Hardy space $H^2(\Pi_r)$. There are two natural ideas to overcome this problem. 

First, for a function $f_{A,r}$ with $A\in \mathcal S$ we can find $A'\in \mathcal S$ such that  $f_{A,r} + f_{A',r} \in K_r^\sharp$ and such that $\Vert f_{A,r} +f_{A',r}-w_\lambda \Vert$ can be controlled by $\Vert f_{A,r}-w_\lambda \Vert$. This strategy is developed through sections \ref{sec:technical-lemma} and \ref{sec:link-distances}. That allows us to state in our main theorem that $d_r^{\sharp}(\lambda) \leq C d_r(\lambda)$ for some explicit constant $C$. With this inequality and  \eqref{eq1:dr-drsharp}, we may use directly the results of  \cite{dfmr}. In particular,  we obtain a Beurling-Nyman's criterion involving $d_r(\lambda)$ for our general class of Dirichlet series (generalizing the previous results of  \cite{Duarte3} and \cite{Anne-BSMF}) and as a by product we also obtain zero free discs (but that are less good than the one in \cite{dfmr}). 

For the second method, we show in Sections~\ref{sec:techniques-explicit} and \ref{sec:zero-free-discs} that we can compensate the pole at $s=1$ by multiplying the function $L(s)$ by a suitable function involving  a Blaschke factor so that the  new function is  in the Hardy space $H^2(\Pi_r)$. This enables us to follow the technics used in \cite{dfmr} to obtain explicit zero free discs. Those new zero free discs improve the ones in \cite{dfmr} and  are easier to describe. Nevertheless, the presence of the Blaschke factor causes some differences and brings some technical calculations; in particular, we must replace the function $w_\lambda$ by  another function $u_{r,\lambda}$ (which lies on $(0,+\infty)$). Then the zero free discs obtained are  expressed in terms of the distance of $u_{r,\lambda}$ to the space $K_r$.

\section{Auxiliary lemmas}\label{sec:technical-lemma}
\subsection{The Pascal matrix}
Let $m\geq 1$ be an integer. The Pascal matrix of size $m\times m$ is defined by 
\begin{equation}\label{Pascal-matrix}
A^{(m)}=(A_{i,j})_{0\leq i,j\leq m-1} \mbox{ with }A_{i,j}=\binom{i+j}{i}. 
\end{equation}
It is known that this is a positive definite symmetric matrix (see \cite[Section 28.4]{Hig}). Hence its greatest eigenvalue $\mu^{(m)}_{max}$ trivially satisfies $\mu^{(m)}_{max}\leq \mathrm{tr}(A^{(m)})$. Moreover, its characteristic polynomial $\chi_m(X)=\det(X I-A^{(m)})$ is palindromic, that is $\chi_m(X)=X^m \chi_m(1/X)$. Then, its lowest eigenvalue $\mu^{(m)}_{min}$ is equal to $1/\mu^{(m)}_{max}$. Moreover, using these two observations and the bound $\binom{2j}{j}\leq 4^j$ on the diagonal coefficients, we get the simple lower bound
\begin{equation}\label{eq:lbpascal}
\mu^{(m)}_{min} \geq \frac {3}{4^m-1}.
\end{equation}
It is also proved in \cite[Section 28.4]{Hig} that 
\[
\mu^{(m)}_{\max}\sim\mathrm{tr}(A^{(m)})\sim \frac{4^{m+1}}{3\sqrt{\pi m}},\qquad m\to +\infty,
\]
which gives the correct order of magnitude of $\mu^{(m)}_{\min}$ as $m$ tends to $+\infty$. For the first values of $m$, we have $\mu_{\min}^{(1)}=1$, $\mu_{\min}^{(2)}=(3-\sqrt{5})/2$, $\mu_{\min}^{(3)}=4-\sqrt{15}$, ... .

\begin{Lem}\label{lem:pascal}
For any $a>0$ and any  $(z_0,z_1,\dots,z_{m-1})\in\CC^{m}$, we have
\[
\int_{0}^{+\infty}\left|\sum_{j=0}^{m-1} z_j\frac{t^j}{j!} \right|^2 e^{-at}\,dt\geq \mu_m\sum_{j=0}^{m-1}\frac{1}{a^{2j+1}}|z_j|^2,
\]
where $\mu_m$ is the lowest eigenvalue of the Pascal matrix defined in \eqref{Pascal-matrix}.
\end{Lem}

\begin{proof}
We restrict the proof to the case $a=1$ since the general case follows from  
\[
\int_{0}^{+\infty}\left|\sum_{j=0}^{m-1} z_j\frac{t^j}{j!} \right|^2 e^{-at}\,dt=\frac{1}{a}\int_{0}^{+\infty}\left|\sum_{j=0}^{m-1} \frac{z_j}{a^j}\frac{t^j}{j!} \right|^2 e^{-t}\,dt.
\]
Expanding the integral and using the identity $\int_0^{+\infty} t^ne^{-t}\,dt=n!$, we have

\[
\int_{0}^{+\infty}\left|\sum_{j=0}^{m-1} z_j\frac{t^j}{j!} \right|^2 e^{-t}\,dt=\sum_{0\leq i,j\leq m-1}\overline{z_i}z_j\frac{(i+j)!}{i!j!}={}^t\bar Z A^{(m)}Z,
\] 
where $Z$ is the column vector ${}^t(z_0,z_1,\dots,z_{m-1})$ and $A^{(m)}$ is the Pascal matrix defined in \eqref{Pascal-matrix}. It remains to note that if $A$ is a hermitian positive definite matrix, then ${}^t\bar Z A Z\geq \mu \sum_{j=0}^{m-1} |z_j|^2$, where $\mu$ is the lowest eigenvalue of~$A$.
\end{proof}
\begin{remark} 
Taking ${}^t(z_0,z_1,\dots,z_{m-1})$ to be an eigenvector for the smallest eigenvalue, we see that the lower bound in the lemma is optimal.
\end{remark}

\subsection{A linear system}

\begin{Lem} \label{syst-triangle}
Let $m\ge 1$, let  $P=(p_0,p_1,\dots,p_{m-1})\in \CC^m$ with $p_{m-1}\neq 0$, and let $\beta=(\beta_0,\dots, \beta_{m-1})\in \CC^{m}$. Then the system
$$
\beta_k=\sum_{i=0}^k\binom{i+m-1-k}{i}p_{i+m-1-k}\thinspace y_i \quad (0\leq k \leq m-1)
$$
of unknown $y=(y_0,\dots,y_{m-1})$ has a unique solution in $\CC^m$ and for such a solution we have
$$
\max_{0\leq k\leq m-1}|y_k|\leq \xi(P)\sum_{k=0}^{m-1}|\beta_k|,
$$
where 
$$
\xi(P)=\begin{cases}
\frac{1}{|p_0|}&\hbox{if }m=1 \\
\frac{1}{|p_{m-1}|}\left(1+\frac{\|P\|_\infty}{|p_{m-1}|} \displaystyle\frac{\left(\frac{m\|P\|_\infty}{|p_{m-1}|}\right)^{m-2}-1}{m\frac{\|P\|_\infty}{|p_{m-1}|}-1}\right)&\hbox{if }m\geq 2,
\end{cases}
$$
and $\|P\|_\infty=\max_ {0\leq i\leq m-1}|p_i|$. 
\end{Lem}

\begin{proof} If $M=(M_{i,j})_{0\le i,j\le m-1}$ is a $m\times m$ matrix with coefficients in $\CC$, we set
$$
\|M\|_{\infty}:=\max_{0\leq i,j\leq m-1}|M_{ij}|.
$$
The system is triangular and the associated matrix $M$ is of the form $M=p_{m-1}D-N$ where $D$ is the diagonal matrix with diagonal coefficients $d_{i,i}=\binom{m-1}{i}$, $0\leq i\leq m-1$, and $N$ is nilpotent and triangular. Since $p_{m-1}\neq 0$, the matrix is invertible so the system has a unique solution and
$$
\max_{0\leq k\leq m-1}|y_k|\leq \|(  p_{m-1}D-N   )^{-1}\|_{\infty}\sum_{k=0}^{m-1}|\beta_k|.
$$
It remains to bound $\|(  p_{m-1}D-N   )^{-1}\|_{\infty}$.
The expected bound is trivial for $m=1$ so we may assume that $m\ge 2$.
We first note that $M=p_{m-1}D(I-\frac{N_1}{p_{m-1}})$, where $N_1=D^{-1}N$. Hence, 
\[
M^{-1}=\frac{1}{p_{m-1}}\sum_{j=0}^{m-1}\frac{N_1^jD^{-1}}{p_{m-1}^j}
\]
and
\begin{equation}\label{eq:M-1}
\|M^{-1}\|_\infty\leq \frac{1}{|p_{m-1}|}\sum_{j=0}^{m-1}\frac{\|N_1^jD^{-1}\|_\infty}{|p_{m-1}|^j}.
\end{equation}
Since $D$ is a diagonal matrix whose diagonal coefficients are at least $1$, we have 
$\|N_1^jD^{-1}\|_{\infty}\leq \|N_1^j\|_\infty$, hence $\|N_1^jD^{-1}\|_{\infty}\leq m^{j-1}\|N_1\|^j_\infty$ if $j\geq 1$. Moreover, the coefficient on the $k$-th row and the $i$-th column in $N_1$ has absolute value 
\[
|p_{i+m-1-k}|\frac{\binom{i+m-1-k}{i}}{\binom{m-1}{k}}=|p_{i+m-1-k}|\prod_{j=0}^{k-i-1}\frac{k-j}{m-1-j}\qquad (k\geq i+1),
\]
which is clearly $\leq |p_{i+m-1-k}|$. Therefore, we get $\|N_1\|_\infty\leq \displaystyle \max_{0\leq i\leq m-1}|p_i|$. Using \eqref{eq:M-1} and setting $q=\max_{0\leq i\leq m-1}|p_i/p_{m-1}|$, we obtain using the fact $q \geq 1$, that 
\begin{eqnarray*}
\|M^{-1}\|_\infty &\leq& \frac{1}{|p_{m-1}|} +\frac{1}{|p_{m-1}|} \sum_{j=1}^{m-1} m^{j-1}q^j \\
                          &=& \frac{1}{|p_{m-1}|} (1+ q\frac{(mq)^{m-2}-1}{mq-1}),
\end{eqnarray*}
which gives the expected result.

\end{proof}

\subsection{A Vandermonde system}
\begin{Lem}\label{lem:vandermonde}
Given a vector $(y_1,y_2,\dots,y_{m})\in\CC^m$, the unique solution of the system 
\begin{equation}\label{eq:vandermonde}
\sum_{j=1}^{m} j^{i-1} x_j=y_i,\qquad (1\leq i\leq m),
\end{equation}
satisfies
\[
\sum_{i=1}^m |x_i|\leq  ((m-1)2^m+1) \max_{1\leq j\leq m} |y_j|.
\]
\end{Lem}
\begin{proof} The result is trivial for $m=1$, since the system then reduces to the equation $x_1=y_1$. We may now assume  $m\ge 2$.
Let $V_m=(j^{i-1})_{1\leq i,j\leq m}$ be the vandermonde matrix  associated to the system. It is known \cite[page 416]{Hig} that the inverse of $V_m$ is given by $W_m=(w_{i,j})_{1\leq i,j\leq m}$, where
\[
w_{i,j}=\frac{(-1)^{m-j} \sigma_{m-j}(1,2,\dots,\widehat i,\dots,m)}{\displaystyle\prod_{\substack{1\leq k\leq m \\ k\neq i}} (i-k)},
\]
 and $\sigma_k$ is the $k$-th symmetric polynomial in $m-1$ indeterminates and where the notation 
 $(1,2,\dots,\widehat i,\dots,m)$ means that we omit the term 
 $i$.   Hence the unique solution of \eqref{eq:vandermonde} satisfies
 \[
 \sum_{i=1}^{m}|x_i|\leq \left(\sum_{1\leq i,j\leq m}|w_{i,j}|\right)  \max_{1\leq j\leq m} |y_j|.
 \]
It remains to prove that $\sum_{1\leq i,j\leq m}|w_{i,j}|=(m-1)2^m+1$. First note that  the denominator of $|w_{i,j}|$ is $(i-1)!(m-i)!$. Moreover, the numerator of $|w_{i,j}|$ is $\sigma_{m-j}(1,2,\dots,\widehat i,\dots,m)$. Then, 
\begin{align*}
\sum_{j=1}^{m}|w_{i,j}|&=\frac{1}{(i-1)!(m-i)!}\sum_{j=0}^{m-1}\sigma_j(1,2,\dots,\widehat i,\dots,m)\\
&=\frac{1}{(i-1)!(m-i)!}\prod_{\substack{1\leq k\leq m\\k\neq i}}(1+k)\\
&=\frac{(m+1)!}{(i-1)!(m-i)!(1+i)}=(m+1)\binom{m}{i}-\binom{m+1}{i+1}.
\end{align*}
By summing the last equality over $1\leq i\leq m$, we get the expected result.
\end{proof}

\section{A Beurling-Nyman's criterion}\label{sec:link-distances}
To simplify the notation, the order $m_L$ of the pole of the Dirichlet series $L$ will be noted $m$ in this section. We recall that $r_0$ is a real number such that $\sigma_0<r_0<1$ and satisfying \eqref{psi-in-L2}. 
~\\
Let $w \in L^2_*((0,1),dt/t^{1-2\sigma_0})$, we consider the distance $\dist(w,K_r)$ and $\dist(w,K_r^\sharp)$, where $K_r$ and $K_r^\sharp$ are defined in \eqref{Kr} and \eqref{Kr-sharp}. We have trivially $\dist(w,K_r) \leq \dist(w,K_r^\sharp)$. We can now state the main result which gives a bound of $\dist(w,K_r^\sharp)$ in function of $\dist(w,K_r)$:
\begin{Thm}\label{maj-dist-sharp} With the previous notation, there exists a positive function  $r\mapsto \theta( \psi,r)$ defined and nonincreasing on  $[r_0,1)$ such that
\begin{equation}\label{eq:distances}
\dist(w,K_r^\sharp) \leq \big(1+\theta(\psi,r)\sqrt{1-r}\big) \dist(w,K_r)
\end{equation}
for each $r\in [r_0,1)$. Furthermore,
$$
\lim_{r\to 1}\theta(\psi,r)\sqrt{1-r}=0.
$$
\end{Thm}
\begin{remark} 
An explicit choice of $\theta(\psi,r)$ will be given in \eqref{choix-theta}, inside the proof of Theorem \ref{maj-dist-sharp}.
\end{remark}
~\\
In the sequel, we will make the following notation. For the  $m$-uplet $P=(p_0,\dots,p_{m-1})$  that has been introduced in    \eqref{psi-u-petit}, we set

\begin{equation}\label{eq:norme-P-2}
\|P\|_{2}:=\left(\int_0^1 \left|\sum_{i=0}^{m-1}p_i(\log u)^i\right|^2 udu\right)^{1/2}.
\end{equation}
Note that $\|P\|_{2}=\left(\int_0^1 |\psi(u)|^2\frac{du}{u}\right)^{1/2}$. Furthermore we set
\begin{equation}\label{eq:norme-psi}
\|\psi\|_{r}=\left(\int_{1}^{+\infty}|\psi(t)|^2\frac{dt}{t^{1+2r}}\right)^{1/2}\qquad (r_0\leq r<1).
\end{equation}
Note that the function $r\mapsto \|\psi\|_r$ is nonincreasing on $[r_0,1)$. Recall that the function $f_{A,r}$ defined   in \eqref{definition_fA} belongs to $L^2((0,+\infty), dt/t^{1-2\sigma_0})$.  For $f\in L^2((0,+\infty), dt/t^{1-2\sigma_0})$, we note 
$$
\Vert f \Vert^2= \int_0^{+\infty} \vert f(t)\vert^2 \frac{dt}{t^{1-2\sigma_0}}.
$$
~\\
Before embarking on the proof of Theorem \ref{maj-dist-sharp}, we need to establish the following crucial lemma.
%
%
%
\begin{Lem}\label{Lem:complete-admi}
Let $A=(\alpha,c) \in \mathcal S$. There exists $A'\in\mathcal S$ such that $f_{A,r}+f_{A',r}\in K_r^\sharp$  
and
\[
\|f_{A,r}+f_{A',r}-w\|\leq \|f_{A,r}-w\|+\Lambda(m,r)\max_{0\leq k\leq m-1}\left|\sum_{j=1}^{\ell(A)} c_j \alpha_j (\log\alpha_j)^k\right|,
\] 
for any $w\in L^2_*((0,1),dt/t^{1-2\sigma_0})$, where 
$$
\Lambda(m,r):=((m-1)2^m+1) \max(e^m,e^{(1-r)m})\left( \|P\|_{2}^2   + \|\psi\|_r^2  \right)^{1/2}.
$$
\end{Lem}
\begin{proof}Let $A=(\alpha,c) \in \mathcal S$. We set
\[
y_k:=\sum_{j=1}^{\ell(A)}c_j\alpha_j(\log \alpha_j)^{k}\quad (0\le k\le m-1).
\]
Our first step is to construct a sequence $A'=(\alpha',c')\in \mathcal{S}$ such that $f_{A,r}+f_{A',r}\in K_r^{\sharp}$.
We choose
$$
\alpha'_j:=e^{-j}\qquad (1\le j\le m).
$$
By Lemma \ref{lem:vandermonde}, there exists a unique $(x_1,\dots,x_m)\in \CC^m$ such that
$$
\sum_{j=1}^{m}x_j(\log \alpha'_j)^k=-y_k\quad (0\le k\le m-1),
$$
and moreover
$$
\sum_{i=1}^m |x_i|\leq  ((m-1)2^m+1) \max_{0\leq k\leq m-1} |y_k|.
$$
Now by choosing
$$
c'_j:=\frac{x_j}{\alpha'_j}\qquad (1\le j\le m)
$$
and using the definition of the $y_k$, we get
\begin{equation}\label{eq:condition-diese}
\sum_{j=1}^{m}c'_j\alpha'_j(\log \alpha'_j)^k+ \sum_{j=1}^{\ell(A)}c_j\alpha_j(\log \alpha_j)^{k}=0\quad (0\le k\le m-1),
\end{equation}
and
\begin{equation*}\label{eq:somme-c'alpha'}
\sum_{i=1}^m |c'_i\alpha'_i|\leq  ((m-1)2^m+1) \max_{1\leq k\leq m} \left| \sum_{j=1}^{\ell(A)}c_j\alpha_j(\log \alpha_j)^{k}   \right|.
\end{equation*}
Moreover, since $0<\alpha'_m\leq\alpha'_j$ for $1\le j\le m$,  this last condition yields
\begin{equation}\label{eq:sum-C'j}
\sum_{j=1}^m|c'_j|\leq e^m ((m-1)2^m+1) \max_{1\leq k\leq m} \left| \sum_{j=1}^{\ell(A)}c_j\alpha_j(\log \alpha_j)^{k}   \right|.
\end{equation}
Now,   by setting $A':=(\alpha',c')$,  the condition   \eqref{eq:condition-diese}    immediately gives
$$
f_{A,r}+f_{A',r}\in K_r^{\sharp},
$$
and in particular  $f_{A,r}(t)+f_{A',r}(t)=0$ for $t>1$ (see \cite[Theorem 4.3]{dfmr}). Furthermore, $w$ is supported on $(0,1)$, hence we have
\begin{align*}
\| f_{A,r}+f_{A',r}-w \|&=  \left(\int_0^1|f_{A,r}(t)+f_{A',r}(t)-w(t)|^2\frac{dt}{t^{1-2\sigma_0}}\right)^{1/2}\\
&\le  \| f_{A,r}-w\| +\left(\int_0^1|f_{A',r}(t)|^2\frac{dt}{t^{1-2\sigma_0}}\right)^{1/2}
\end{align*}
from which we deduce
\begin{equation}\label{eq:fA}
\| f_{A,r}+f_{A',r}-w   \|\le  \| f_{A,r}-w\|+\sum_{j=1}^{m}|c'_j|\left(\int_0^1|\psi(\tfrac{\alpha'_j}{t})|^2\frac{dt}{t^{1-2r}}\right)^{1/2}.
\end{equation}
Taking \eqref{eq:sum-C'j} into account, the lemma follows immediately from the bounds
$$
\int_0^1|\psi(\tfrac{\alpha'_j}{t})|^2\frac{dt}{t^{1-2r}}\le \max(1,e^{-2rm})\left(\|P\|_2^2    +   \|\psi\|_r^2 \right)
$$
for each $1\le j\le m$. For proving these inequalities, we write
\begin{align*}
\int_0^1|\psi(\tfrac{\alpha'_j}{t})|^2\frac{dt}{t^{1-2r}}&=(\alpha'_j)^{2r}\int_{\alpha'_j}^{+\infty}|\psi(t)|^2\frac{dt}{t^{1+2r}}\\
&=  (\alpha'_j)^{2r}\int_{\alpha'_j}^{1}|\psi(t)|^2\frac{dt}{t^{1+2r}}+(\alpha'_j)^{2r} \|\psi\|_r^2.
\end{align*}
Considering the cases $r>0$ and $r\leq 0$, we have
$$
\frac{(\alpha'_j)^{2r}}{t^{2r}}\le \max\big(1,e^{-2rm}\big)\quad (1\le j\le m,\thinspace \alpha'_j\le t\le 1),
$$
which gives 
$$
(\alpha'_j)^{2r}\int_{\alpha'_j}^{1}|\psi(t)|^2\frac{dt}{t^{1+2r}}\le \max\big(1,e^{-2rm}\big)\|P\|_2^2.
$$
With the same method, we obtain
$$
(\alpha'_j)^{2r} \|\psi\|_r^2\le \max\big(1,e^{-2rm}\big) \|\psi\|_r^2
$$
which concludes the proof.
\end{proof}
\bigskip
~\\
\textit{Proof of Theorem \ref{maj-dist-sharp}}. \medskip
We set
\[
E(r):=\left(2\sum_{k=0}^{+\infty}\frac{(2-2r)^{2k}}{(k!)^2} \right)^{1/2},
\]
where $\|P\|_{2}$ and $\|\psi\|_r$ have been introduced in \eqref{eq:norme-P-2} and \eqref{eq:norme-psi} respectively.
We also denote by $\mu_m$  the lowest eigenvalue of the Pascal matrix defined in \eqref{Pascal-matrix}. \medskip
~\\
We are now ready to prove Theorem \ref{maj-dist-sharp} with
\begin{equation}\label{choix-theta}
\theta(\psi,r):=\frac{\xi(\psi)E(r)\Lambda(m,r)}{\sqrt{\mu_m}}.
\end{equation}
where $\xi(\psi):=\xi(P)$ is defined in Lemma~\ref{syst-triangle} and where $\Lambda(m,r)$ is defined in Lemma~\ref{Lem:complete-admi}. It is clear that with this choice, the function $r\mapsto \theta(\psi,r)$ is nonincreasing on $[r_0,1)$. In particular, $\theta(\psi,r)$ is bounded on $[r_0,1)$ and then $\lim_{r\to 1}\theta(\psi,r)\sqrt{1-r}=0$.  
\bigskip
~\\
Let  $A=(\alpha,c)\in \mathcal{S}$. Assume that the following inequality
\begin{equation}\label{eq:estimation-max}
\max_{0\leq k\leq m-1}\left|\sum_{j=1}^{\ell(A)} c_j \alpha_j (\log\alpha_j)^k\right|\leq \|f_{A,r}-w\|\xi(\psi) E(r)\sqrt{\frac{1-r}{\mu_m}}
\end{equation}
holds. Then we complete  the proof  of Theorem \ref{maj-dist-sharp} as follows: 
 according to Lemma~\ref{Lem:complete-admi}, there exists $A'\in\mathcal S$ such that $f_{A,r}+f_{A',r}\in K_r^\sharp$ and 
\[
\|f_{A,r}+f_{A',r}-w\|\leq \|f_{A,r}-w \|+\Lambda(m,r)\max_{0\leq k\leq m-1}\left|\sum_{j=1}^{\ell(A)} c_j \alpha_j (\log\alpha_j)^k\right|,
\]
for any $w\in L^2((0,1),dt/t^{1-2\sigma_0})$. Hence, using \eqref{eq:estimation-max} and the inequality $\dist(w,K_r^\sharp)\le \|f_{A,r}+f_{A',r}-w\|$, we have
$$
\dist(w,K_r^\sharp) \leq \left(1+ \xi(\psi) E(r)\Lambda(m,r)\sqrt{\frac{1-r}{\mu_m}}    \right) \|f_{A,r}-w\|, 
$$
and  \eqref{eq:distances} follows immediately by     taking the infimum over $A\in\mathcal S$.
~\\
It remains to prove \eqref{eq:estimation-max}.
Since $w(t)=0$ for $t>1$, we have
$$
\int_1^{+\infty}\left|\sum_{j=1}^{\ell(A)}c_j\psi\left(\tfrac{\alpha_j}{t}\right)\right|^2\frac{dt}{t^{1-2r}} =\int_1^{+\infty}|f_{A,r}(t)|^2\frac{dt}{t^{1-2\sigma_0}}\leq \|f_{A,r}-w\|^2.
$$
Hence, using the notation in \eqref{psi-u-petit}, one has
\begin{align*}
\|f_{A,r}-w\|^2 & \geq \int_1^{+\infty}\left|\sum_{j=1}^{\ell(A)}c_j\psi\left(\tfrac{\alpha_j}{t}\right)\right|^2\frac{dt}{t^{1-2r}}\\
&=\int_1^{+\infty}\left|\sum_{j=1}^{\ell(A)}c_j\frac{\alpha_j}{t}\sum_{i=0}^{m-1}p_i\big(\log \alpha_j-\log t\big)^i\right|^2\frac{dt}{t^{1-2r}}\\
&=\int_1^{+\infty}\left|\sum_{k=0}^{m-1}(-1)^{m-1-k}(\log t)^{m-1-k}\beta_k\right|^2\frac{dt}{t^{3-2r}}\\
&=\int_0^{+\infty}\left|\sum_{k=0}^{m-1}(-1)^{m-1-k}u^{m-1-k} \beta_k\right|^2 e^{-2(1-r)u}\,du.
\end{align*}
where we have set for $0\le k\le m-1$
$$
\beta_k:=\sum_{i=0}^k\binom{i+m-1-k}{i}p_{i+m-1-k}\thinspace \sum_{j=1}^{\ell(A)}c_j\alpha_j(\log \alpha_j)^{i}.
$$
Lemma \ref{syst-triangle} then gives 
$$
\max_{0\le k\le m-1}\left| \sum_{j=1}^{\ell(A)}c_j\alpha_j(\log \alpha_j)^{k}  \right|\le \xi(\psi)\sum_{k=0}^{m-1}|\beta_k|.
$$
Now, Cauchy's inequality yields
$$
\left(\sum_{k=0}^{m-1}|\beta_k|\right)^{2}\le \left(\sum_{k=0}^{m-1}\frac{(2-2r)^{2k+1}}{(k!)^2}\right)\left(\sum_{k=0}^{m-1}\frac{|k!\beta_k|^2}{(2-2r)^{2k+1}}\right), 
$$
and using Lemma \ref{lem:pascal} with the choice $z_k=k!\beta_k$ and $a=2-2r$, one has
$$
\sum_{k=0}^{m-1}\frac{|k!\beta_k|^2}{(2-2r)^{2k+1}}\leq \frac{1}{\mu_m} \int_0^{+\infty}\left|\sum_{k=0}^{m-1}(-1)^{m-1-k}u^{m-1-k} \beta_k\right|^2 e^{-2(1-r)u}\,du.
$$
Then we deduce
\begin{align*}
\sum_{k=0}^{m-1}|\beta_k|&\le \frac{1}{\sqrt{\mu_m}}\left( \sum_{k=0}^{m-1}\frac{(2-2r)^{2k+1}}{(k!)^2}\right)^{1/2}\|f_{A,r}-w\|.\\
&\le E(r)\sqrt{\frac{1-r}{\mu_m}}\|f_{A,r}-w\|.\\
\end{align*}
This ends the proof of \eqref{eq:estimation-max}. \hfill $\square$
\bigskip 
~\\
Now, we apply Theorem \ref{maj-dist-sharp} with $w(t)=w_\lambda(t)= t^{\overline{\lambda}-2\sigma_0}\chi_{(0,1)}(t)$ where $\lambda\in\Pi_{\sigma_0}$. In that case, we denote $\dist(w_\lambda,K_r)$ (resp. $\dist(w_\lambda,K_r^\sharp)$) by $d_r(\lambda)$ and $d_r^\sharp(\lambda)$ (see \eqref{definition_dr} and \eqref{eq1:dr-drsharp}).

\begin{Cor}\label{cor_maj-dist-sharp} With the previous notation, we have
\begin{equation}\label{eq:distances2}
d_r(\lambda)  \leq d_r^\sharp (\lambda) \leq \big(1+\theta(\psi,r)\sqrt{1-r}\big) d_r(\lambda) 
\end{equation}
for all $r\in [r_0,1)$ and $\lambda \in \Pi_{\sigma_0}$.
\end{Cor}
~\\
Theorem \ref{maj-dist-sharp} and Corollary \ref{cor_maj-dist-sharp} give directly the Beurling-Nyman criterion  we had in mind (i.e. without the technical conditions):
\begin{Cor}\label{beurling-nyman}
Let $r_0\leq r<1$.  Assume  that  $\hat\varphi$ does not vanish on the half-plane $\Pi_{r}$, that 
$\limsup_{x\to+\infty}\frac{\log|\hat\varphi(x+r-\sigma_0)|}{x} = 0$ and that $a_1 \neq 0$.  Then the following assertions are equivalent:
\begin{enumerate}[(1)]
\item The function $L$  does not vanish on the half-plane $\Pi_{r}$.
\item There exists $\lambda\in\Pi_{\sigma_0}$ such that $d_r(\lambda)=0$.
\item For all $\lambda\in\Pi_{\sigma_0}$, we have $d_r(\lambda)=0$.
\item $L^2_*((0,1),dt/t^{1-2\sigma_0}) \subset K_r$. 
\end{enumerate}
\end{Cor}

\begin{proof}
According to Corollary~\ref{cor_maj-dist-sharp}, we have 
\[
d_r(\lambda)=0 \Longleftrightarrow d_r^\sharp(\lambda)=0.
\]
Hence, the equivalence between the first three assertions follows immediately from Theorem~\ref{avec_conditions}. The implication $(4)\Longrightarrow (3)$ is trivial. The remaining implication $(1)\Longrightarrow (4)$ comes again from Theorem~\ref{avec_conditions} and the fact that $K_r^\sharp\subset K_r$.
\end{proof}
~\\
As already mentioned, this Beurling-Nyman's criterion generalizes previous results obtained in \cite{Duarte3} and \cite{Anne-BSMF} for the Riemann zeta function and the Selberg class. As an illustration, take a Dirichlet series $L(s)=\sum_{n\geq 1} a_{n} n^{-s}$ in the Selberg class with $a_{1}\neq 0$ (otherwise the Dirichlet series is zero by the multiplicative properties of $a_n$). Then 
$L(s)$ has an analytic continuation to $\CC \setminus \{1 \}$  and we choose $\sigma_{0}=0$. Let $d$ be the degree of $L$ and take 
$$
\varphi(t)= \frac{\chi_{(0,1)}(t)}{(1-t)^{\sigma_{1}}} \mbox{ where } \sigma_{1} < \frac{1}{2} -\frac{d}{4}.
$$
Then, we may choose $r_0=1/2$ so that $\psi \in L^2((1,\infty), \frac{du}{u^{1+2r_0}})$ (see \cite[Section 7.3]{dfmr}). Moreover, for 
$\lambda=r=1/2$, the function $t^{\bar\lambda}\chi_{(0,1)}\in  K_{r}$ if and only if the function $\chi_{(0,1)}$ belongs to the 
space 
$$
\mathcal B_{\sigma_{1}} =\Span  \{ t \mapsto \sum_{j=1}^{\ell(\alpha)} c_{j} \psi\left(\frac{\alpha_{j}}{t}\right) \colon (\alpha,c)\in \mathcal S \},
$$
where the closed linear span here is taken with respect to the space $L^2((0,\infty),dt)$.
Hence:
\begin{Cor} 
The function $L(s)$ does not vanish for $\Re(s)>1/2$ if and only if $\chi_{(0,1)} \in \mathcal B_{\sigma_{1}}$.
\end{Cor}

\section{The function $f_{A,r}$ when $A\in\mathcal S$}\label{sec:techniques-explicit}
From this section, we investigate an other method for compensating the pole at $s=1$ coming from $L(s)$. 

Recall that we have fixed a real number $r_0>\sigma_0$ satisfying \eqref{psi-in-L2} and then for any $r$ such that $r_0\leq r<1$, we have $\psi \in L^2((0,+\infty), \frac{du}{u^{1+2r}})$. Recall also that for any $A=(\alpha,c)\in\mathcal S$ and $r_0\leq r<1$, the function $f_{A,r}$, defined by 
$$
f_{A,r}(t)=t^{r-\sigma_0} \sum_{j=1}^{\ell(\alpha)} c_j \psi\left(\frac{\alpha_j}{t}\right) , \quad (t >0),
$$
belongs to the space $L^2((0,+\infty), \frac{dt}{t^{1-2\sigma_0}})$. If $f\in L^2((0,+\infty), dt/t^{1-2\sigma_0})$, we note
$$
\Vert f \Vert^2= \int_0^{+\infty} \vert f(t)\vert^2 \frac{dt}{t^{1-2\sigma_0}}.
$$
\begin{Lem}\label{Lem:fonction-fA-non-admissible}
Let $r_0\leq r<1$ and $A=(\alpha,c)\in\mathcal S$. Then 
\begin{enumerate}[(a)]
\item We have
\[
\lim_{\substack{r\to r_0\\ 
r>r_0}}\|f_{A,r}-f_{A,r_0}\|=0.
\]
\item The integral
\[
\int_0^{+\infty}f_{A,r}(t)t^{s-1}\,dt
\]
is absolutely convergent if $\sigma_0+r_0-r<\Re(s)<\sigma_0+1-r$.
\end{enumerate}
\end{Lem}
\medskip 
\begin{proof} 
For the first point, note that
\begin{equation}\label{eq:trick}
f_{A,r}(t)=t^{r-r_0}f_{A,r_0}(t) = t^{r-r_1}f_{A,r_1}(t) \qquad (r_0\leq r,r_1<1),
\end{equation}
which proves that $f_{A,r}(t)$ tends pointwise to $f_{A,r_0}(t)$ on $(0,+\infty)$, as $r\to r_0$. Moreover, if $r_1$ is such that $r<r_1<1$, then, using the two equalities in \eqref{eq:trick}  (depending whether $t<1$ or $t\geq1$), we easily check that 
\[
|f_{A,r}(t)|\leq |f_{A,r_0}(t)|+|f_{A,r_1}(t)|,\qquad t>0.
\]
Since the function $t\mapsto |f_{A,r_0}(t)|+|f_{A,r_1}(t)|$ is in $L^2((0,+\infty),\frac{dt}{t^{1-2\sigma_0}})$, an application of Lebesgue's theorem gives the result.
~\\
For the second point, by linearity and using a change of variable, it is sufficient to prove that the integral
\[
\int_0^{+\infty}\frac{|\psi(t)|}{t^{\sigma+r+1-\sigma_0}}\,dt
\]
is convergent if $\sigma_0+r_0-r<\sigma<\sigma_0+1-r$ which is equivalent to the convergence of 
\[
\int_0^{+\infty}\frac{|\psi(t)|}{t^{1+\gamma}}\,dt
\]
if  $r_0<\gamma<1$. On the one hand, $\psi(t)=tP(\log t)$ for $t\in (0,1)$ and we have 
\[
\int_0^1\frac{|\psi(t)|}{t^{1+\gamma}}\,dt=\int_0^1 \frac{|P(\log t)|}{t^{\gamma}}\,dt.
\]
This last integral is convergent if $\gamma<1$. On the other hand, using Cauchy--Schwarz inequality, we get 
\[
\int_1^{+\infty}\frac{|\psi(t)|}{t^{1+\gamma}}\,dt\leq \left(\int_1^{+\infty}\frac{|\psi(t)|^2}{t^{1+2r_0}}\,dt \right)^{1/2} \left(\int_1^{+\infty}\frac{dt}{t^{1+2\gamma-2r_0}} \right)^{1/2}.
\]
Now the first integral on the right hand side is finite by hypothesis and the second integral is finite if and only if $r_0<\gamma$, which concludes the proof.
\end{proof}
~\\
If $A=(\alpha,c)\in\mathcal S$, we let 
$$
g_A(s)=\sum_{j=1}^{\ell(\alpha)}c_j\alpha_j^s \quad (s\in\CC).
$$

\begin{Lem}\label{Lem:mellin--transform-nnadmi}~Let $r_0\leq r<1$ and $s \in \CC$ with $ \sigma_0+r_0-r<\Re(s) <\sigma_0+1-r.$
Then
\begin{equation}\label{eq:mellin-transform-fA}
\widehat{f_{A,r}}(s)=-L(s+r-\sigma_0)\widehat{\varphi}(s+r-\sigma_0)g_A(s+r-\sigma_0).
\end{equation}
If $r=r_0$ and $s=\sigma_0+it$, the equality \eqref{eq:mellin-transform-fA} holds for almost every $t \in \RR$.
\end{Lem}
\begin{proof}
For $r_0<\Re(s)<1$, we claim that 
\begin{equation}\label{eq:Mellintransformpsi(1surt)}
-L(s)\hat\varphi(s)=\int_0^{+\infty}\psi\left(\frac{1}{t}\right)t^{s-1}dt.
\end{equation}
Indeed , on one hand, by \cite[Lemma 3.1]{dfmr}, we have 
\begin{equation}\label{eq:cle-Lem-mellin--transform-nnadmi}
H(s)=\int_0^1\psi\left(\frac{1}{t}\right)t^{s-1}\,dt\qquad (\Re(s)>1),
\end{equation}
where $H$ is the analytic function on $\Pi_{\sigma_0}$ introduced in \eqref{eq:defnH}. 
Since the function $t\mapsto \psi(\frac{1}{t})\chi_{[0,1]}(t)$ belongs to $L^2_*((0,1),
\frac{dt}{t^{1-2r_0}})$, the function $s\mapsto\int_0^1 \psi\left(\frac{1}
{t}\right)t^{s-1}\,dt$ is analytic on $\Pi_{r_0}$. Hence,the analytic continuation 
principle implies that the equality \eqref{eq:cle-Lem-mellin--transform-nnadmi} is 
satisfied for all $s\in\Pi_{r_0}$. On the other hand, by an easy induction argument, we 
have
\[
\frac{1}{(k-1)!}\int_0^1(\log t)^{k-1}t^{-s}\,dt=-\frac{1}{(s-1)^k}\qquad (\Re(s)<1).
\]
Thus, using \eqref{psi-u-petit}, we get
\begin{equation}\label{eq:psi(1surt)2}
\int_0^1\psi(t)t^{-s-1}\,dt=-\sum_{k=0}^{m_L-1}\frac{k!p_k}{(s-1)^{k+1}}\qquad (\Re(s)<1).
\end{equation}
Using \eqref{eq:cle-Lem-mellin--transform-nnadmi} and \eqref{eq:psi(1surt)2}, we obtain,  for $r_0<\Re(s)<1$, 
\begin{eqnarray*}
-L(s)\hat\varphi(s)&=&H(s)-\sum_{k=0}^{m_L-1}\frac{k!p_k}{(s-1)^{k+1}}\\
&=&\int_0^1\psi\left(\frac{1}{t}\right)t^{s-1}\,dt+\int_0^1\psi(t)t^{-s-1}\,dt,
\end{eqnarray*}
which yields \eqref{eq:Mellintransformpsi(1surt)}. Therefore 
\begin{eqnarray*}
-L(s)\hat\varphi(s)g_A(s)&=&\int_0^{+\infty}\psi\left(\frac{1}{t}\right) \sum_{j=1}^{\ell(\alpha)}c_j\alpha_j^s t^{s-1}\,dt\\
&=&\int_0^{+\infty}\sum_{j=1}^{\ell(\alpha)}c_j \psi\left(\frac{\alpha_j}{t}\right)t^{s-1}\,dt\\
&=&\int_0^{+\infty} f_{A,r}(t)t^{s+\sigma_0-r-1}\,dt.
\end{eqnarray*}
Lemma~\ref{Lem:fonction-fA-non-admissible} (b) implies that the last integral is absolutely convergent if $r_0<\Re(s)<1$. Hence
\begin{equation}\label{eq2:mellin-transform-fA}
-L(s)\hat\varphi(s)g_A(s)= \widehat{f_{A,r}}(s+\sigma_0-r),\qquad r_0<\Re(s)<1.
\end{equation}
We conclude the proof of \eqref{eq:mellin-transform-fA} using the change of variable $s\longmapsto s-\sigma_0+r_0$.
~\\
For the second part, take $r$ such that $r_0<r<1$. Now, we know that
\begin{equation}\label{eq:toujours-pas-didee}
\widehat{f_{A,r}}(\sigma_0+it)=-L(r+it)\widehat{\varphi}(r+it)g_A(r+it).
\end{equation}  
By Lemma~\ref{Lem:fonction-fA-non-admissible} (a), the sequence $\widehat{f_{A,r_n}}$ tends to $\widehat{f_{A,r_0}}$ in $L^2(\sigma_0+i\RR)$, for any sequence $(r_n)_n$ tending to $r_0$ (since the Mellin transform is an isometry from $L^2((0,\infty),dt/t^{1-2\sigma_0})$ onto $L^2(\sigma_0+i\RR)$). Using a classical result, this sequence $(r_n)_n$ can be chosen so that 
\[
\lim_{n\to +\infty}\widehat{f_{A,r_n}}(\sigma_0+it)=\widehat{f_{A,r_0}}(\sigma_0+it)
\]
for almost all $t\in\RR$. The equation \eqref{eq:toujours-pas-didee} is now sufficient to complete the proof. \end{proof}
\medskip
~\\
We need to fix some other notation. If $r \in \RR$ and $\lambda\in\Pi_r$, we denote 
by $k_{\lambda,r}$ (respectively by $b_{\lambda,r}$) the reproducing kernel of $H^2(\Pi_r)$ (respectively the elementary Blaschke factor of $H^2(\Pi_r)$) corresponding to the point $\lambda$. In others words, we have for $k_{\lambda,r}$
\begin{equation}\label{eq:nr-H2r}
k_{\lambda,r}(s)=\frac{1}{2\pi}\frac{1}{s-2r+\bar\lambda},\qquad s\in\Pi_r,
\end{equation}
and
\begin{equation}\label{eq:bf-H2r}
f(\lambda)=\langle f,k_{\lambda,r}\rangle_2=\frac{1}{2\pi} \int_{-\infty}^{+\infty}\frac{f(r+it)}{\lambda-r-it}\,dt,
\end{equation}
for any function $f\in H^2(\Pi_r)$.
We also have for $b_{\lambda,r}$
\begin{equation}\label{eq:blaschke-elementaire-H2r}
b_{\lambda,r}(s)=\frac{s-\lambda}{s+\bar\lambda-2r},\qquad s\in\Pi_r,
\end{equation}
which is analytic and bounded on the closed half-plane $\overline{\Pi_r}$.  More precisely, we have $|b_{\lambda,r}(s)|\leq 1$ if $s\in \Pi_r$ and $|b_{\lambda,r}(s)|=1$ if $\Re(s) =r$. \medskip

\begin{Lem}\label{lem:appartenance-H2-multiplication-blaschke}
Let $A=(\alpha,c)\in\mathcal S$. Then for all $r$, $r_0\leq r<1$, the function 
$$
s\mapsto L(s)\widehat{\varphi}(s)g_A(s)b_{1,r}^{m_L}(s)
$$
belongs to $H^2(\Pi_{r})$.
\end{Lem}

\begin{proof}
Recall that 
\[
L(s)\widehat{\varphi}(s)=\sum_{k=0}^{m_L-1}\frac{k!p_k}{(s-1)^{k+1}}-H(s),\qquad s\not=1,s\in\Pi_{\sigma_0},
\]
and according to the proof of Theorem 2.1 in \cite{dfmr}, the function $H$ belongs to $H^2(\Pi_{r})$. Therefore
\[
L(s)\widehat{\varphi}(s)g_A(s)b_{1,r}^{m_L}(s)=\sum_{k=0}^{m_L-1}k!p_k\frac{b^{m_L}_{1,r}(s)}{(s-1)^{k+1}}g_A(s)-H(s)g_A(s)b^{m_L}_{1,r}(s).
\]
The functions $g_A$ and $b^{m_L}_{1,r}$ are bounded on $\Pi_{r}$, hence the function $s\mapsto H(s)g_A(s)b^{m_L}_{1,r}(s)$ belongs to $H^2(\Pi_{r})$. So it is sufficient to prove that the function
$$
s\longmapsto \frac{b^{m_L}_{1,r}(s)}{(s-1)^{k+1}}
$$ 
belongs to $H^2(\Pi_{r})$  for every $0\leq k\leq m_L-1$. Using \eqref{eq:blaschke-elementaire-H2r} and the fact that $|b_{1,r}(s)|\leq 1$ for $s\in\Pi_{r}$, we have 
\[
\left|\frac{b_{1,r}^{m_L}(s)}{(s-1)^{k+1}}\right|\leq \left|\frac{b_{1,r}^k(s)}{(s-1)^{k+1}}\right|=\frac{1}{|s+1-2r|^{k+1}},
\]
and it is clear that 
\begin{eqnarray*}
\sup_{\sigma=\Re(s)>r}\left(\int_{-\infty}^{+\infty}\frac{dt}{|\sigma+1-2r+it|^{2(k+1)}}\right)&\leq& \int_{-\infty}^{+\infty}\frac{dt}{((1-r)^2+t^2)^{k+1}} \\
&<&+\infty,
\end{eqnarray*}
which proves that $s\longmapsto \frac{b^{m_L}_{1,r}(s)}{(s-1)^{k+1}}$ belongs to $H^2(\Pi_{r})$.
\end{proof}
~\\
We introduce now a function $u_{r,\lambda}$ of $L^2((0,\infty),dt/t^{1-2\sigma_0})$ which will be used to give explicit zero free regions in terms of the distance of $u_{r,\lambda}$ to the subspace $K_r$ (see \eqref{eq:defdelta} and Theorem~\ref{mauv_dist}). This function $u_{r,\lambda}$ plays the role of the function $w_\lambda$ in Theorem~\ref{sans_zeros_avec_conditions}.  
For $\lambda \in \Pi_{\sigma_0}$ and $t>0$, we define
$$
u_{r,\lambda}(t)= \left(1+\frac{A}{B}\right)^{m_L} t^{\overline{\lambda}-2\sigma_0} \chi_{(0,1)}(t)
 + Q_{r,\lambda}(\log t) t^{r-\sigma_0 -1} \chi_{(1,\infty)}(t) 
$$
with $A=2-2r$ and $B=r+\sigma_0-1-\overline{\lambda}$ and where $Q_{r,\lambda}$ is the polynomial defined by 
$$
Q_{r,\lambda}(t) = -  \sum_{j=0}^{m_L-1} \left(\sum_{k=0}^{m_L-1-j} \binom{m_L}{k} \left(\frac{A}{B}\right)^{m_L-k}\right) \frac{(-1)^{j} B^j}{j!} t^j .
$$
Note that for $0<t<1$, we have $u_{r,\lambda}(t)=(1+A/B)^{m_L}w_\lambda(t)$, but the 
function $u_{r,\lambda}$ is (contrary to the function $w_\lambda$) supported on the 
whole axis $(0,\infty)$. This is quite natural since $K_r$ is formed by functions which 
live on $(0,\infty)$ whereas functions of $K_r^\sharp$ vanish on $(1,\infty)$. 
Although the formulae defining $u_{r,\lambda}$ may appear a little bit complicated, 
this function is chosen so that its Mellin transform has the simple following form:
\begin{Lem}\label {lemme_polynome} For $2\sigma_0-\Re(\lambda)<\Re(s)<1+\sigma_0-r$, we have
$$
\widehat{u_{r,\lambda}}(s) = 2\pi \frac{k_{\lambda,\sigma_0}(s)}{b_{1,r}^{m_L}(s+r-\sigma_0)}.
$$
\end{Lem}
\begin{proof} We compute the inverse Mellin transform of the right hand side  which is defined by 
$$
\frac{1}{2i\pi} \int_{Re(z)=\delta} 2\pi \frac{k_{\lambda,\sigma_0}(z)}{b_{1,r}^{m_L}(z+r-\sigma_0)} t^{-z} dz
$$
where $\delta$ is any real number such that $2\sigma_0-\Re(\lambda) < \delta < 1+\sigma_0-r$.\smallskip 
~\\
If $0<t<1$, we  can push the line of integration to the left. Then we catch the residue of the function $2\pi\frac{k_{\lambda,\sigma_0}(z)}{b_{1,r}^{m_L}(z+r-\sigma_0)} t^{-z}$ at the simple pole $z=2\sigma_0-\overline{\lambda}$ and we obtain $u_{r,\lambda}(t)$ for $0<t<1$.\smallskip
~\\
If $t>1$, we can push the line of integration to the right. We catch the residue of the function at the pole $z=\sigma_0+1-r$ which is of order $m_L$. After a direct computation, we obtain $u_{r,\lambda}(t)$ for $t>1$.
\end{proof}

\section{Zero-free regions}\label{sec:zero-free-discs}
For $A\in\mathcal S$, we set
$$
h_{A,r}(s)=-\frac{1}{\sqrt{2\pi}}L(s+r-\sigma_0)\hat\varphi(s+r-\sigma_0)g_A(s+r-\sigma_0)b_{1,r}^{m_L}(s+r-\sigma_0),
$$
and according to Lemma \ref{lem:appartenance-H2-multiplication-blaschke}, the function  $h_{A,r}$ belongs to $H^2(\Pi_{\sigma_0})$. Furthermore, $h_{A,r}$ is analytic in $\Pi_{2\sigma_0-r}$.
\begin{Prop}\label{sans_zeros} Let $r_0$ satisfying \eqref{psi-in-L2}, let $r_0\leq r<1$, and let $\lambda \in \Pi_{\sigma_0}$. Then $L$ does not vanish on 
$$
r-\sigma_0 +\left\{ \mu \in \CC \colon \left| \frac{ \mu -\lambda}{\mu+\overline{\lambda} -2\sigma_0}\right| < \sqrt{4\pi\big(\Re(\lambda)-\sigma_0\big)} \frac{|h_{A,r}(\lambda)|}{\Vert h_{A,r} \Vert_2 }\right\},
$$
for any $A\in\mathcal S$.
\end{Prop}
~\\
\begin{remark} Before proving the proposition, we should recall that for $\lambda=a+ib\in\Pi_{\sigma_0}$ ($a>\sigma_0,b\in\RR$) and $R\in [0,1]$, the set
\[
\left\{\mu\in\CC:\left|\frac{\mu-\lambda}{\mu+\bar\lambda-2\sigma_0}\right|<R\right\}
\]
is the open (euclidean) disc whose center is $\Omega=\left(\frac{a+R^2(a-2\sigma_0)}{1-R^2},b\right)$ and radius is $\frac{2R(a-\sigma_0)}{1-R^2}$ if $R\in [0,1[$; if $R=1$ this set is the half-plane $\Pi_{\sigma_0}$.
\end{remark}
\begin{proof}[Proof of Proposition \ref{sans_zeros}.]
Using Lemma~\ref{lem:appartenance-H2-multiplication-blaschke} and following the proof of \cite[Corollary 2.3]{dfmr} we obtain that $L$ does not vanish on the disc
$$
r-\sigma_0 +\left\{ \mu \in \CC \colon \left| \frac{ \mu -\lambda}{\mu+\overline{\lambda} -2\sigma_0}\right| < \frac{|h_{A,r} (\lambda) |}{\Vert h_{A,r} \Vert_2 \Vert k_\lambda\Vert_2} \right\}
$$
where the norms $\Vert \cdot \Vert_2$ are relative to the space $H^2(\Pi_{\sigma_0})$ (see equation~\eqref{norm_H2}). It remains to note that $\Vert k_\lambda \Vert_2 = \left(4\pi\big(\Re(\lambda)-\sigma_0\big)\right)^{-1/2}$.
\end{proof}
\begin{remark}\label{remark_sans_zero}
For applications, in order to compute $\frac{|h_{A,r}(\lambda)|}{\Vert h_{A,r} \Vert_2 }$, it is useful to notice that $\Vert h_{A,r} \Vert_2 = \Vert f_{A,r}\Vert_2$.
Indeed using the fact that $h_{A,r}$ is not only in $H^2(\Pi_{\sigma_0})$ but is analytic on $\Pi_{2\sigma_0-r}$ and using that $|b_{1,r}^{m_L}(s)|=1$ if 
$\Re(s)=r$, we obtain 
\[
\Vert h_{A,r} \Vert_2^2 = \int_{-\infty}^{\infty} \left|\frac{1}{\sqrt{2\pi}} L(r +it)\hat\varphi(r+it)g_A(r+it)\right|^2 dt. 
\]
Here we used the classical fact (see \cite{duren}) for example) that, if a function $h$ is in $H^2(\Pi_{\sigma_0})$, then the limit $h(\sigma_0+it):=\displaystyle\lim_{\substack{\sigma\to\sigma_0\\>}}h(\sigma+it)$ exists for almost all $t\in\RR$ and 
\[
\|h\|_2^2=\int_{-\infty}^{+\infty}|h(\sigma_0+it)|^2\,dt.
\]
 
By Lemma~\ref{Lem:mellin--transform-nnadmi}, we get
\[
\Vert h_{A,r} \Vert_2^2 = \int_{-\infty}^{\infty} \left|\frac{1}{\sqrt{2\pi}} \widehat{f_{A,r}}(\sigma_0 +it) \right|^2 dt. 
\]
The result now follows since ${\mathcal M}$ is a unitary operator from the space $L^2((0,\infty),dt/t^{1-2\sigma_0})$ onto $L^2(\sigma_0+i\RR)$.
\end{remark}
~\\
For $\lambda \in \Pi_{\sigma_0}$, we let
\begin{equation}\label{eq:defdelta}
\delta_r(\lambda) = \dist(u_{r,\lambda}, K_r),
\end{equation}
where the distance is taken with respect to $L^2((0,\infty), dt/t^{1-2\sigma_0})$.

\begin{Thm}\label{mauv_dist} Let $r_0$ satisfying \eqref{psi-in-L2}, let $r_0\leq r<1$, and let $\lambda \in \Pi_{\sigma_0}$. Then $L$ does not vanish on $r-\sigma_0+D_r(\lambda)$ where 
$$
D_r(\lambda) := \left\{ \mu \in \CC \colon \left|\frac{\mu - \lambda}{\mu + \overline{\lambda}-2\sigma_0}\right| < \sqrt{1-2(\Re(\lambda) -\sigma_0)\delta_r^2(\lambda)} \right\} .
$$
\end{Thm}
\medskip
\begin{proof}
According to proposition~\ref{sans_zeros}, the function $L$ does not vanish on 
$$
r-\sigma_0+\left\{ \mu \in \CC \colon \left|\frac{\mu - \lambda}{\mu + \overline{\lambda}-2\sigma_0}\right| < R \right\}
$$
with
$$
R=\frac{1}{\Vert k_{\lambda,\sigma_0} \Vert_2} \sup_{A\in\mathcal S} \frac{|h_{A,r}(\lambda)|}{\Vert h_{A,r} \Vert_2}.
$$
~\\
It remains to prove that $R^2=1-2(\Re(\lambda) -\sigma_0)\delta_r^2(\lambda)$. For this, note that
\begin{eqnarray*}
R &=& \frac{1}{\Vert k_{\lambda,\sigma_0}\Vert_2} \sup_{A\in\mathcal S} \frac{|\langle h_{A,r} , k_{\lambda,\sigma_0}\rangle|}{\Vert h_{A,r} \Vert_2}\\
   &= & \frac{1}{\Vert k_{\lambda,\sigma_0}\Vert_2} \Vert P_{E_r} k _{\lambda,\sigma_0} \Vert_2
\end{eqnarray*}
where $E_r = \Span_{H^2(\Pi_{\sigma_0})} (h_{A,r} \colon A\in\mathcal S)$. Hence, with standard Hilbert space arguments, we get
\begin{equation}\label{eq:formule-rayon-abstrait}
R^2=1-\frac{\dist^2(k_{\lambda,\sigma_0},E_r)}{\Vert k_{\lambda,\sigma_0} \Vert_2^2}.
\end{equation}
Take any $A=\in\mathcal S$. Using $|b_{1,r}^{m_L}(r+it)|=1$, we obtain by Lemma \ref{Lem:mellin--transform-nnadmi} and Lemma \ref{lemme_polynome} that
\begin{multline*}
\Vert k_{\lambda,\sigma_0} - h_{A,r} \Vert_2^2 = \\  \int_{-\infty}^\infty \left| \frac{k_{\lambda,\sigma_0}(\sigma_0 +it)}{b_{1,r}^{m_L}(r+it)}+\frac{1}{\sqrt{2\pi}} L(r+it) \widehat{\varphi}(r+it) g_{A,r}(r+it) \right|^2 dt \\
=  \int_{-\infty}^\infty \left| \frac{1}{2\pi}\widehat{u_{r,\lambda}}(\sigma_0+it) - \frac{1}{\sqrt{2\pi}}\widehat{f_{A,r}}(\sigma_0+it)\right|^2 dt.
\end{multline*}
Since the Mellin transform is a unitary map, we deduce 
$$
\Vert k_{\lambda,\sigma_0} - h_{A,r} \Vert_2^2 =  \Vert \frac{1}{\sqrt{2\pi}}u_{r,\lambda} - f_{A,r} \Vert_2^2.
$$
Thus $\dist^2(k_{\lambda,\sigma_0},E_r) =  \frac{1}{2\pi}\dist^2(u_{r,\lambda}, K_r)$. Now the desired equality follows from \eqref{eq:formule-rayon-abstrait} and the fact that $\|k_{\lambda,\sigma_0}\|_2=(4\pi(\Re(\lambda)-\sigma_0)^{-1/2}$.
\end{proof}
\begin{remark}
If $u_{r,\lambda}\in K_r$ for some $\lambda\in\Pi_{\sigma_0}$ and some $r\geq r_0$, then it follows immediately from Theorem~\ref{mauv_dist} that $L$ does not vanish on $\Pi_r$ (indeed in this case $\delta_r(\lambda)=0$, and then the zero free region obtained in Theorem~\ref{mauv_dist} is exactely the half-plane $\Pi_r$). We do not know if the converse is true. 
\end{remark}

\begin{remark} The strategy of using the Blaschke factor in order to kill the pole of $L$ is also successful to prove the implication ${\it (2)} \Rightarrow {\it (1)}$ in Corollary \ref{beurling-nyman} (indeed this is the key implication to prove). For this, we can follow and generalize the idea given in \cite[Th\'eor\`eme II]{anne-TAMS}. Suppose that $L(s_0)=0$ for some $s_0\in\Pi_{r}$. Then consider 
\[
\mathfrak{u}(s)=\overline{b_{1,r}^{m_L}(s+r-\sigma_0)}k_{s_0,r}(s+r-\sigma_0), \quad (\Re(s)\geq \sigma_0).
\]
Note that $\mathfrak{u}$ has roughly the same flavor as $\widehat{u_{r,\lambda}}$. First we have 
\[
|\mathfrak{u}(\sigma_0+it)|=|b_{1,r}^{m_L}(r+it)| |k_{s_0,r}(r+it)|=|k_{s_0,r}(r+it)|, 
\]
and $t\longmapsto k_{s_0,r}(r+it)$ belongs to $L^2(\RR)$, so $\mathfrak{u}\in L^2(\sigma_0+i\RR)$. We claim that 
\begin{equation}\label{eq1:cle-BN-ss-admissibilite}
\int_{-\infty}^{+\infty}\widehat{f_{A,r}^\sharp}(\sigma_0+it)\overline{\mathfrak{u}(\sigma_0+it)}\,dt=0
\end{equation}
and that
\begin{equation}\label{eq2:cle-BN-ss-admissibilite}
\int_{-\infty}^{+\infty}k_{\lambda,\sigma_0}(\sigma_0+it)\overline{\mathfrak{u}(\sigma_0+it)}\,dt\not=0
\end{equation}
for every $\lambda\in\Pi_{\sigma_0}$.  

Let us prove \eqref{eq1:cle-BN-ss-admissibilite}. By Lemma~\ref{Lem:mellin--transform-nnadmi} and Lemma~\ref{lem:appartenance-H2-multiplication-blaschke}, we obtain 
\begin{multline}\label{equation_dep}
\int_{-\infty}^{+\infty} \widehat{ f_{A,r}}(\sigma_0+it)\overline{\mathfrak{u}(\sigma_0+it)}\,dt= \\
-\int_{-\infty}^{+\infty}L(r+it)\hat\varphi(r+it)g_A(r+it)b^{m_L}_{1,r}(r+it)\overline{k_{s_0,r}(r+it)}\,dt
\end{multline}
and the last integral is simply
$$
-\langle L\hat\varphi g_A b^{m_L}_{1,r},k_{s_0,r}\rangle_{H^2(\Pi_{r})} =-L(s_0)\hat\varphi(s_0)g_A(s_0) b^{m_L}_{1,r}(s_0) = 0.
$$
Finally let us prove \eqref{eq2:cle-BN-ss-admissibilite}. We have
$$
\int_{-\infty}^{+\infty}k_{\lambda,\sigma_0}(\sigma_0+it)\overline{\mathfrak{u}(\sigma_0+it)}\,dt = \langle h,k_{s_0,r}\rangle_{H^2(\Pi_{r})},
$$
where $h(s)=k_{\lambda,\sigma_0}(s+\sigma_0-r)b_{1,r}^{m_L}(s)$. Since $k_{\lambda,\sigma_0}\in H^2(\Pi_{\sigma_0})$, the function $s\longmapsto k_{\lambda,\sigma_0}(s+\sigma_0-r)$ belongs to $H^2(\Pi_{r})$, so $h\in H^2(\Pi_{r})$. We deduce
\[
\langle h,k_{s_0,r}\rangle_{H^2(\Pi_{r})}=h(s_0)=k_{\lambda,\sigma_0}(s_0+\sigma_0-r)b_{1,r}^{m_L}(s_0)\not=0,
\]
because $s_0\not=1$. That concludes the proof of \eqref{eq2:cle-BN-ss-admissibilite}.\medskip
~\\
Hence, we have constructed a function $\mathfrak{u}\in L^2(\sigma_0+i\RR)$ which is orthogonal to all functions $\widehat{f_{A,r}}$ but not orthogonal to $k_{\lambda,\sigma_0}$. Therefore 
\[
k_{\lambda,\sigma_0}\not\in\hbox{span}(\widehat{f_{A,r}} :A \in \mathcal S). 
\]
A direct calculation gives $\sqrt{2\pi}k_{\lambda,\sigma_0}=\mathcal M(t^{\overline{\lambda}-2\sigma_0}\chi_{(0,1)}(t))$ and since $\mathcal M$ is an isometry from $L^2((0,+\infty),\frac{du}{u^{1-2\sigma_0}})$ onto $L^2(\sigma_0+i\RR)$, we get that 
\[
t^{\bar\lambda-2\sigma_0}\chi_{(0,1)}\not\in \hbox{span}(f_{A,r} :A \in \mathcal S)=K_{r},
\]
which contradicts ${\it (2)}$ and concludes the proof. 
\end{remark}
\begin{remark} There are 3 distances involved through this work and  \cite{dfmr}, they are:
\begin{itemize}
\item The distance $\delta_r(\lambda) = \dist(u_{r,\lambda}, K_r)$;
\item The distance $d_r(\lambda) = \dist(t^{\overline{\lambda}-2\sigma_0}\chi_{(0,1)}, K_r)$;
\item The distance $d_r^\sharp(\lambda) := \dist(t^{\overline{\lambda}-2\sigma_0}\chi_{(0,1)}, K_r^\sharp)$.
\end{itemize}
We have:
$$
d_r(\lambda) \leq d_r^\sharp(\lambda) \leq C  d_r(\lambda)
$$
for some $C$. The first inequality is trivial and the second one is Corollary \ref{cor_maj-dist-sharp}. Unfortunately,
a direct comparison between $\delta_r(\lambda)$ and $d_r(\lambda)$ or $d_r^\sharp(\lambda)$ seems to be not obvious to obtain and so we are not able to compare directly the zero-free discs appearing in Theorem \ref{mauv_dist} and in Theorem \ref{sans_zeros_avec_conditions}. 
\end{remark}

\section{Explicit applications}

Let $L(s) = \sum_{n\geq 1} a_n n^{-s}$ be a Dirichlet series satisfying the conditions in section \ref{section_2}. 
We make use of Proposition \ref{sans_zeros} (and of Remark \ref{remark_sans_zero}). Let us recall that for $A=(\alpha,c) \in \mathcal S$ and $r_0\leq r<1$, the function $f_{A,r}$, defined by
$$
f_{A,r}(t) = t^{r-\sigma_0} \sum_{j=1}^{\ell(\alpha)} c_j \psi\left(\frac{\alpha_j}{t}\right),
$$ 
belongs to $L^2((0,\infty),dt/t^{1-2\sigma_0})$. For numerical applications, we have to estimate the norm of $f_{A,r}$. From the definition of $\psi$ in \eqref{psi}, we let
$$
\psi_1(u)=\res\left(L(s)\hat\varphi(s)u^s,s=1\right).
$$
The computations in \cite[equation (4.4)]{dfmr} give the following upper bound for the norm of $f_{A,r}(t)$ in $L^2((0,+\infty),dt/t^{1-2\sigma_0})$
$$
\Vert f_{A,r} \Vert \leq \sum_{j=1}^{\ell(\alpha)} |c_j| \alpha_j^r \left( \Vert \psi_1 \Vert_{L^2((0,1),\frac{du}{u^{1+2r}})} + \Vert \psi \Vert_r\right).
$$
Recall that by definition (see \eqref{eq:norme-psi}), we have $\Vert \psi \Vert_r = \Vert \psi \Vert_{L^2((1,\infty),\frac{du}{u^{1+2r}})}$.
Note that we also have the better theoretical upper bound
$$
\Vert f_{A,r} \Vert_{L^2((0,\infty),\frac{dt}{t^{1-2\sigma_0}})} \leq \Vert \psi \Vert_{L^2((0,\infty),\frac{du}{u^{1+2r}})}\sum_{j=1}^{\ell(\alpha)} |c_j| \alpha_j^r, $$ 
but it is less pratical for numerical applications. Hence, we deduce from Proposition \ref{sans_zeros} and Remark \ref{remark_sans_zero} that $L$ does not vanish in the disc
$$
r-\sigma_0 +\left\{ \mu \in \CC \colon \left|\frac{\mu-\lambda}{\mu+\overline{\lambda}-2\sigma_0}\right| < \sqrt{2(\Re(\lambda)-\sigma_0)} \;R \right\},
$$
where
$$
R= \frac{\left|L(\lambda+r-\sigma_0) \widehat{\varphi}(\lambda+r-\sigma_0)  b_{1,r}^{m_L}(\lambda+r-\sigma_0)\sum_{j=1}^{\ell(\alpha)} c_j \alpha_j^{\lambda+r-\sigma_0}\right|}{ \sum_{j=1}^{\ell(\alpha)} |c_j| \alpha_j^r \left(  \Vert \psi_1 \Vert_{L^2((0,1),\frac{du}{u^{1+2r}})} + \Vert \psi \Vert_r\right)}.
$$
In particular, for $\ell(\alpha)=1$ and $c_1=\alpha_1=1$, we obtain 
\begin{Prop} With the notation above, $L$ does not vanish on the disc
$$
r-\sigma_0 +\left\{ \mu \in \CC \colon \left|\frac{\mu-\lambda}{\mu+\overline{\lambda}-2\sigma_0}\right| < \sqrt{2(\Re(\lambda)-\sigma_0)}\; R \right\},
$$
where 
$$
R=
 \frac{\left|L(\lambda+r-\sigma_0) \widehat{\varphi}(\lambda+r-\sigma_0)  b_{1,r}^{m_L}(\lambda+r-\sigma_0)\right|}{ \Vert \psi_1 \Vert_{L^2((0,1),\frac{du}{u^{1+2r}})} + \Vert \psi \Vert_r}.
 $$
\end{Prop}
\subsection{Zero-free discs for $\zeta$}
We  apply the proposition above to the case of the Riemann zeta function $\zeta(s)$.
~\\
We choose $\sigma_0=0$, $\varphi(t)=(1-t)^{-\sigma_1}\chi_{(0,1)}$ with 
$$
\sigma_1<1/2 \mbox{ and } \max(0,\sigma_1) < r < 1.
$$ 
~\\
We have
$$
\widehat{\varphi}(s)= \frac{\Gamma(s)\Gamma(1-\sigma_1)}{\Gamma(1+s-\sigma_1)}.
$$
In \cite[Theorem 7.2] {dfmr}, it is obtained 
$$
\Vert \psi \Vert_r^2  \leq \frac{1}{2r} + C(\sigma_1) \zeta(1+2(r-\sigma_1)),
$$
where 
$$
C(\sigma_1) =\frac{1}{1-2\sigma_1}+\frac{1}{(1-\sigma_1)^2(3-2\sigma_1)}+\frac{\varepsilon_1}{(1-\sigma_1)^2},
$$
with $\varepsilon_1=1$ if $\sigma_1\geq 0$ and $\varepsilon_1=-1$ if $\sigma_1<0$. Furthermore, in our case $\psi_1(u)=u/(1-\sigma_1)$ so 
$$
\Vert \psi_1\Vert_{L^2((0,1),\frac{du}{u^{1+2r}})}^2 = \frac{1}{(1-\sigma_1)^2(2-2r)}.
$$
~\\
Hence
\begin{Thm} The Riemann zeta function does not vanish in the disc
$$
r+ \left\{\mu \in \CC \colon \left|\frac{\mu-\lambda}{\mu+\bar{\lambda}}\right| < F(\lambda,r,\sigma_1) \right\},
$$
where 
\begin{multline*}
F(\lambda,r,\sigma_1)= \\ 
\frac{\sqrt{2\Re(\lambda)} \left| \Gamma(\lambda+r) \Gamma(1-\sigma_1) \zeta(\lambda+r)(\lambda+r-1) \right|}{ \left(C(r,\sigma_1)+\frac{1}{(1-\sigma_1)\sqrt{2-2r}}\right) \big|\Gamma(\lambda+r+1-\sigma_1)(\lambda-r+1) \big|},
\end{multline*}
~\\
with $C(r,\sigma_1)=\sqrt{1/(2r)+C(\sigma_1)\zeta(1+2r-2\sigma_1)}$.
\end{Thm}
~\\
It seems to be more convenient than the zero-free discs that are given in \cite[Corollary 7.4]{dfmr} (at least, we do not have to optimize any more the choice of $A \in \mathcal S^\sharp$). As an example, taking 
$\lambda=0.01+50i$, $r=0.49$ and $\sigma_1=0.4$ (these are the same values taken in the numerical example in \cite{dfmr}), we obtain that $\zeta$ does not vanish in the disc of center $\frac{1}{2}+50i$ and radius $1.49\times 10^{-5}$ (remark that it is a little bit better than in \cite{dfmr} in which the radius is $5.13\times 10^{-6}$). Note also that we did not optimize the choices of the parameters $\lambda, r$ and $\sigma_1$(and indeed, the choices of the parameters are not the best in order to get a zero-free region around $\frac{1}{2} + 50i$).   


\bibliographystyle{alpha}
\bibliography{bibli}

\begin{thebibliography}{BDBLS00}

\bibitem[BDBLS00]{Duarte3}
Luis B{\'a}ez-Duarte, Michel Balazard, Bernard Landreau, and Eric Saias.
\newblock Notes sur la fonction {$\zeta$} de {R}iemann. {III}.
\newblock {\em Adv. Math.}, 149(1):130--144, 2000.

\bibitem[Beu55]{Beurling}
A.~Beurling.
\newblock A closure problem related to the riemann zeta-function.
\newblock {\em Proc. Nat. Acad. USA}, 41(5):312--314, 1955.

\bibitem[DFMR11]{dfmr}
C.~Delaunay, E.~Fricain, E.~Mosaki, and O.~Robert.
\newblock Zero free regions for dirichlet series.
\newblock 2011.
\newblock To appear in Transactions of the American Mathematical Society.

\bibitem[dR06]{Anne-BSMF}
A.~de~Roton.
\newblock Une approche hilbertienne de l'hypoth\`ese de {R}iemann
  g\'en\'eralis\'ee.
\newblock {\em Bull. Soc. Math. France}, 134(3):417--445, 2006.

\bibitem[dR07]{anne-TAMS}
A.~de~Roton.
\newblock G\'en\'eralisation du crit\`ere de {B}eurling-{N}yman pour
  l'hypoth\`ese de {R}iemann.
\newblock {\em Trans. Amer. Math. Soc.}, 359(12):6111--6126 (electronic), 2007.

\bibitem[dR09]{Anne-JNT}
A.~de~Roton.
\newblock Une approche s\'equentielle de l'hypoth\`ese de {R}iemann
  g\'en\'eralis\'ee.
\newblock {\em J. Number Theory}, 129(11):2647--2658, 2009.

\bibitem[Dur70]{duren}
P.L. Duren.
\newblock {\em Theory of {$H\sp{p}$} spaces}.
\newblock Pure and Applied Mathematics, Vol. 38. Academic Press, New York,
  1970.

\bibitem[Hig02]{Hig}
N.J. Higham.
\newblock {\em Accuracy and Stability of Numerical Algorithms}.
\newblock Second Edition. Siam, 2002.

\bibitem[Nik95]{Nikolski-AIF}
N.~Nikolski.
\newblock Distance formulae and invariant subspaces, with an application to
  localization of zeros of the {R}iemann {$\zeta$}-function.
\newblock {\em Ann. Inst. Fourier (Grenoble)}, 45(1):143--159, 1995.

\bibitem[Nym50]{Nyman}
B.~Nyman.
\newblock {\em {On some groups and semi-groups of translations}}.
\newblock PhD thesis, Upsala, 1950.

\end{thebibliography}

\end{document}